\documentclass{amsproc}

\usepackage{amsmath, amssymb, amsfonts, amsthm, enumitem, comment, url}
\usepackage{xcolor}
\theoremstyle{definition}
\newtheorem{defn}{Definition}

\theoremstyle{plain}
\newtheorem{thm}{Theorem}
\newtheorem{lemma}[thm]{Lemma}
\newtheorem{prop}[thm]{Proposition}

\theoremstyle{remark}
	\newtheorem*{rmk}{Remark}
	\newtheorem*{example}{Example}

\newcommand{\cA}{\mathcal{A}}
\newcommand{\cB}{\mathcal{B}}
\newcommand{\cF}{\mathcal{F}}

\newcommand{\cT}{\mathcal{T}}

\newcommand{\N}{\mathbb{N}}
\newcommand{\R}{\mathbb{R}}
\newcommand{\Z}{\mathbb{Z}}

\newcommand{\fT}{\mathfrak{T}}

\newcommand\ShVar[1]{\left\lVert#1\right\rVert_{\mathrm{ShVar}}}
\newcommand{\ShReg}{\mathrm{ShReg}}
\newcommand\VolVar[1]{\left\lVert#1\right\rVert_{\mathrm{VolVar}}}
\newcommand{\VolReg}{\mathrm{VolReg}}

\title{Conformal measures and the Dobrushin-Lanford-Ruelle equations}
\author{Lu{\'i}sa Borsato}
\address{Departamento de Matem\'atica, Instituto de Matem\'atica e Estat\'istica, Universidade de S\~ao Paulo,
R. do Mat\~ao 1010, S\~ao Paulo, SP 05508-900, Brazil}
\email{luisabb@ime.usp.br}
\thanks{The first author is supported by grants 2018/21067-0 and 2019/08349-9, S\~ao Paulo Research Foundation (FAPESP)}

\author{Sophie MacDonald}
\address{Mathematics Department, University of British Columbia, Vancouver, British Columbia, Canada, V6T 1Z2}
\email{sophmac@math.ubc.ca}


\begin{document}

\begin{abstract}
We demonstrate the equivalence of two definitions of a Gibbs measure on a subshift over a countable group. We formulate a more general version of the classical Dobrushin-Lanford-Ruelle equations with respect to a measurable cocycle, which reduce to the classical equations when the cocycle is induced by an interaction or a potential, and show that a measure satisfying these equations must have the conformal property. We also review methods of constructing an interaction from a potential and vice versa, such that the interaction and the potential have the same Gibbs and equilibrium measures. 
\end{abstract}

\maketitle

\section{Introduction}

This paper is concerned with two notions of a Gibbs measure on a subshift over a countable group. The first of these is defined by the Dobrushin-Lanford-Ruelle (DLR) equations, or equivalently a Gibbsian specification. This notion of a Gibbs measure appears for instance in the classical theorems of Dobrushin \cite{dobrushin1970conditional} and Lanford-Ruelle \cite{lanfordruelle1969observables}. The second is the notion of a conformal measure, introduced in \cite{petersen1997symmetric} and 
\cite{denker-urbanski-1991-conformal} and used for instance by Meyerovitch in \cite{meyerovitch-2013-gibbs-eqm} as the setting for a stronger Lanford-Ruelle theorem. There are other definitions in the literature, such as a Gibbs measure in the sense of Bowen, but we do not consider these here.

The purpose of the present article is to show that the two notions of Gibbs measure recalled above coincide in some generality. Our results build on those of Kimura, obtained in his Master thesis \cite{kimura-2015-thesis}, who proves two results relevant here. The first is that every conformal measure, with respect to an appropriately regular potential, satisfies the DLR equations for that potential. The second is a partial converse, namely that every measure satisfying the DLR equations for such a potential is topologically Gibbs. This is a weaker property than being conformal, although equivalent on certain subshifts, such as shifts of finite type \cite{meyerovitch-2013-gibbs-eqm}. Sarig (\cite{sarig2009notes}, Proposition 2.1) shows that the two notions of Gibbsianness are equivalent in the case of a topologically mixing one-sided shift of finite type, using martingale and Ruelle transfer operator methods (note that he uses the word ``conformal'' for a notion that is related to but different from the one we consider). Cioletti-Lopes-Stadlbauer (\cite{cioletti2020ruelle}) prove a similar result for one-sided, one-dimensional shifts with quite general alphabet. Muir \cite{muir2011paper} also obtains the full equivalence for the full shift on $\Z^d$ over a countable alphabet.

Our main result, Theorem \ref{main-thm-abs}, strengthens one of Kimura's results in a more general setting. Specifically, we show that any measure satisfying certain equations with respect to a measurable cocycle on the Gibbs relation must also be conformal with respect to that cocycle. When the cocycle is induced either by an interaction or by a potential in the standard way, these equations reduce to the classical DLR equations. We prove this result for arbitrary subshifts with finite alphabet on an arbitrary countable group. The results of Kimura and Sarig in the forward direction (conformal implies DLR) can also be generalized to our setting; in \S \ref{equiv_sec}, we mention the idea for the proof but refer readers to \cite{kimura-2015-thesis} for the details in Kimura's setting, as the proof strategy changes very little.

The plan is as follows. In \S \ref{def_sec}, we review the definitions and basic facts required to prove our main result in \S \ref{equiv_sec}. In \S \ref{intrxn_sec} and \S \ref{potl_sec}, we recall well-known material on interactions and potentials, respectively, in order to show that the equations involved in our main theorem do in fact reduce to the classical DLR equations. In \S \ref{potl_from_intrxn_sec}, we recall results of Muir and Kimura, elaborating on Ruelle, by which a potential can be constructed from a sufficiently regular interaction, and vice versa, with ``physical'' data (Gibbs and equilibrium measures) preserved. 

In \S \ref{potl_sec} and \S \ref{potl_from_intrxn_sec}, we require that the underlying group admits a finite generating set that yields a certain spherical growth condition, defined in \S \ref{potl_sec}. This condition is satisfied, for any generating set, by any group of polynomial growth of nilpotency class at most $2$, such as $\Z^d$, the case of greatest physical interest. It is also satisfied by any free group $F_n$, with the usual generating set of cardinality $n$, and, conditional on a folklore conjecture, is satisfied by any nilpotent group.

\section{Cocycles and the Gibbs relation: definitions and properties}\label{def_sec}

Throughout, let $G$ be a countable group with identity $e$. Let $\cA$ be a finite alphabet equipped with the discrete topology, and $X \subseteq \cA^G$ a subshift, i.e., a closed set in the product topology, invariant under the shift action of $G$ via $(g \cdot x)_h = x_{g^{-1} h}$. The topology on $X$ is generated by cylinders, i.e., sets of the form $[\omega] = \{ x \, | \, x_{\Lambda} = \omega \}$ for finite sets $\Lambda \Subset G$. We use the notation $\Lambda \Subset G$ to indicate that $\Lambda$ is a finite subset of $G$. This topology can be induced by a metric such that the resulting metric space is complete and separable; that is, $\cA^G$ is a Polish space. We equip $X$ with the Borel $\sigma$-algebra $\cF$.

The \textit{Gibbs relation}, also called the asymptotic relation, is the equivalence relation $\fT_X \subset X \times X$ such that $(x,y) \in \fT_X$ if and only if $x_{\Lambda^c} = y_{\Lambda^c}$ for some finite set $\Lambda \Subset G$. Let $(\Lambda_N)_{N=1}^{\infty}$ be a sequence of finite sets exhausting $G$, i.e., $(\Lambda_N)_{N=1}^{\infty}$ is an increasing sequence and $G = \displaystyle\cup_{N=1}^{+\infty} \Lambda_N$. Define the subrelation $\fT_{X,N} = \{ (x,y): x_{\Lambda_N^c} = y_{\Lambda_N^c} \} \subseteq \fT_X$. Observe that, for each subrelation $\fT_{X,N}$, each equivalence class is a finite set, and that $\fT_X = \cup_{N=0}^{\infty} \fT_{X,N}$. (In the language of Borel equivalence relations, this means that $\fT_X$ is \textit{hyperfinite} \cite{kechris2019borel}, which we mention for context, although we do not use any theorems about hyperfiniteness in this paper.) In particular, every equivalence class in $\fT_X$ is at most countable. Note that we can write each subrelation as $\fT_{X,N} = \cap_{n=N}^{\infty} \cup_{\omega \in \cA^{\Lambda_n \setminus \Lambda_N}} [\omega] \times [\omega]$, which shows that $\fT_{X,N}$ is a measurable subset of $X \times X$ in the product $\sigma$-algebra $\cF \otimes \cF$, as is $\fT_X$.

For Borel sets $A, B \subseteq X$, a \textit{holonomy} of $\fT_X$ is a Borel isomorphism $\psi: A \to B$ such that $(x, \psi(x)) \in \fT_X$ for all $x \in A$. We say that a holonomy $\psi$ is \textit{global} if $A = B = X$. The definitions for $\fT_{X,N}$ are analogous, with a holonomy of $\fT_{X,N}$ also a holonomy of $\fT_X$, for every $N$.

For a Borel set $A \subseteq X$, we denote $\fT_X(A) = \cup_{x \in A} \{ y \in X | (x,y) \in \fT_X \}$, and the same for the subrelations. The saturations $\fT_X(A)$ and $\fT_{X, N}(A)$ are easily shown to be Borel using the fact that the diagonal in $X \times X$ is measurable in the product $\sigma$-algebra, which follows as an easy exercise from the fact that $X$ is Polish.

\begin{lemma}\label{group}
There exists a countable group $\Gamma$ of global holonomies of $X$ such that
\begin{equation*} 
\fT_X = \{ (x, \gamma(x)): x \in X, \gamma \in \Gamma \}.
\end{equation*}
In other words, $\Gamma$ generates $\fT_X$.
\end{lemma}

\begin{proof}
The group $\Gamma$ can be described explicitly as a countable increasing union of finite groups $\Gamma_N$. For each $N$, the group $\Gamma_N$  generates $\fT_{X, N}$ and is isomorphic to the symmetric group of order $|\cA^{\Lambda_N}|$. Take $\Gamma_N$ to be generated by holonomies $\psi$ of the following form: given $\omega, \eta \in \cA^{\Lambda_N}$, define $\psi_{\omega, \eta}: X \to X$ by
\begin{equation*}
    \psi_{\omega, \eta}(x) = \begin{cases}
    \eta x_{\Lambda_N^c} & x_{\Lambda_N} = \omega, \, \eta x_{\Lambda_N^c} \in X \\
    \omega x_{\Lambda_N^c} & x_{\Lambda_N} = \eta, \, \omega x_{\Lambda_N^c} \in X \\
    x & \text{otherwise}
    \end{cases}
\end{equation*}
That is, $\psi_{\omega, \eta}$ exchanges $\omega$ and $\eta$, wherever possible, and otherwise does nothing. These involutions were considered in \cite{meyerovitch-2013-gibbs-eqm} and \cite{kimura-2015-thesis}, for slightly different purposes.

Observe that $(x,y) \in \fT_{X,N}$ if and only if there exists $\psi \in \Gamma_N$ with $\psi(x) = y$, so $\fT_{X,N}$ is precisely the orbit relation of $\Gamma_N$. The result for $\Gamma$ is immediate.
\end{proof}

\begin{rmk}
We mention for context that Lemma \ref{group} is a special case of the main theorem of \cite{feldman-moore-1977-equivalence}, which in fact asserts the same for any Borel equivalence relation on a Polish space in which every equivalence class is countable. This result was adapted to the symbolic setting in \cite{meyerovitch-2013-gibbs-eqm}, with the countability of the equivalence classes established via the expansivity of the shift action. The proof is presented for subshifts over $\Z^d$, but the same proof goes through for arbitrary countable groups without modification. However, since we establish Lemma \ref{group} directly, we do not need to appeal to the theorem of \cite{feldman-moore-1977-equivalence} (nor the symbolic corollary in \cite{meyerovitch-2013-gibbs-eqm}).
\end{rmk}

We say that a measure $\mu$ on $X$ (by which we always mean a Borel probability measure) is $\fT_X$-nonsingular if for every Borel $A \subset X$ with $\mu(A) = 0$, we have $\mu(\fT_X(A)) = 0$. Note that if $\mu$ is $\fT_X$-nonsingular and $\psi: A \to B$ is a holonomy of $\fT_X$, then whenever $E \subset A$ has $\mu(E) = 0$, we have $\mu(\psi(E)) \leq \mu(\fT_X(E)) = 0$. In particular, the Radon-Nikodym derivative $\frac{d(\mu \circ \psi)}{d\mu}$ is well-defined. The same holds with $\fT_X$ replaced by $\fT_{X,N}$.

A (real-valued) cocycle on $\fT_X$ is a Borel measurable function $\phi: \fT_X \to \R$ such that $\phi(x,y) + \phi(y,z) = \phi(x,z)$ for all $x,y,z \in X$ with $(x,y), (y,z) \in \fT_X$ (so that $(x,z) \in \fT_X$ as well). Any cocycle on $\fT_X$ clearly restricts to a cocycle on $\fT_{X,N}$, for any given $N$. Given a $\fT_X$-nonsingular measure $\mu$ on $X$, we say that a Borel function $D: \fT_X \to \R$ is a \textit{Radon-Nikodym cocycle} on $\fT_X$ with respect to $\mu$ if the pushforward of $\mu$ by any holonomy $\psi: A \to B$ of $\fT_X$ satisfies $\frac{d(\mu \circ \psi)}{d\mu}(x) = D(x, \psi(x))$ for $\mu$-a.e. $x \in A$. It is routine to show, using Lemma \ref{group}, that any $\fT_X$-nonsingular measure $\mu$ on $X$ has a $\mu$-a.e. unique Radon-Nikodym cocycle. 

\begin{defn}[conformal measure]
    Let $\mu$ be a $\fT_X$-nonsingular Borel probability measure on $X$, and let $\phi: \fT_X \to \R$ be a cocycle. We say that $\mu$ is $(\phi, \fT_X)$-\textit{conformal} if for any holonomy $\psi: A \to B$ of $\fT_X$, with $A$ and $B$ Borel sets, we have
    \[
    \mu(B) = \int_A \exp(\phi(x,\psi(x))) \, d\mu(x)
    \]
\end{defn}

Note that this is equivalent to the condition that
\[
D_{\mu, \fT_X}(x,\psi(x)) = \exp(\phi(x,\psi(x)))
\]
for $\mu$-a.e. $x \in A$. Note also that a $\fT_X$-nonsingular measure is conformal precisely with respect to the logarithm of its Radon-Nikodym cocycle.

\begin{rmk} The name ``conformal measure'' was given to a related kind of measure in \cite{denker-urbanski-1991-conformal}, motivated by Patterson's study \cite{patterson1976fuchsian} of measures on the limit sets of particular groups of conformal mappings of the unit disc in the complex plane. The term ``Gibbs measure'' was first applied to the present measures in \cite{capocaccia-1976-definition}.
\end{rmk}

\begin{defn}[DLR equations for a cocycle]
Let $X \subseteq \cA^G$ be a subshift, $\phi$ a cocycle on $\fT_X$, and $\mu$ a measure on $X$. For a Borel set $A \subseteq X$ and a finite set $\Lambda \Subset G$, the DLR equation for $x \in X$ is as follows:

\begin{equation}\label{dlr-eq-cocycle}
    \mu(A \,|\, \mathcal{F}_{\Lambda^c})(x) =
    \sum_{\eta \in \cA^{\Lambda}} \left[ \sum_{\zeta \in \cA^{\Lambda}} \exp( \phi(\eta x_{\Lambda^c}, \zeta x_{\Lambda^c})) \mathbf{1}_X(\zeta x_{\Lambda^c})   \right]^{-1} \mathbf{1}_A(\eta x_{\Lambda^c})
\end{equation}
We say that $\mu$ is DLR with respect to $\phi$ if, for any Borel $A \subseteq X$ and any $\Lambda \Subset G$, \eqref{dlr-eq-cocycle} holds for $\mu$-a.e. $x \in X$.
\end{defn}

\section{Equivalence of the conformal and DLR properties}\label{equiv_sec}

For us, the main value of Lemma \ref{group} is the following lemma, which reveals in particular that to show that a given measure is conformal (such as in Theorem \ref{main-thm-abs}), it is sufficient to consider only global holonomies.

\begin{lemma}\label{ets_group_conf}
Let $\mu$ be a Borel probability measure on $X$, let $\phi$ be a cocycle on $\fT_X$, and let $\Gamma$ be a countable group generating $\fT_X$. Then $\mu$ is $(\phi, \fT_X)$-conformal if and only if, for each $\gamma \in \Gamma$, the pushforward $\mu \circ \gamma$ is absolutely continuous with respect to $\mu$, with $\frac{d(\mu \circ \gamma)}{d\mu}(x) = \exp(\phi(x, \gamma(x))$ for $\mu$-a.e. $x \in X$.
\end{lemma}

\begin{proof}
The ``only if'' direction is immediate from the definition of conformal measure. To confirm the ``if'' direction, we first check nonsingularity. Let $A \subset X$ be Borel with $\mu(A) = 0$. Then $\fT_X(A) = \bigcup_{\gamma \in \Gamma} \gamma(A)$, which is a countable union and thus has measure zero by the explicit expression for $\frac{d (\mu\circ\gamma)}{d\mu}$.

Now let $\psi: A \to B$ be a holonomy of $\fT_X$ and let $E \subseteq A$ be Borel. Let $\Gamma = (\gamma_n)_{n \in \N}$ be an enumeration of $\Gamma$. For each $n \in \N$, let $E_n = \{ x \in E: \psi(x) = \gamma_n(x) \}$. To see that each $E_n$ is Borel, define the map $\tau_n: X \to X \times X$ by $\tau_n(x) = (\psi(x), \gamma_n(x))$, which is clearly measurable in the product $\sigma$-algebra. Then $E_n = \tau_n^{-1}(D)$ where $D \subset X \times X$ is the diagonal, which, as discussed above, is also Borel in the product $\sigma$-algebra, because $X$ is Polish.


Now let $E_0' = E_0$, and for $n \geq 1$, let $E_n' = E_n \setminus \cup_{k=0}^{n-1} E_k$. The Borel sets $E_n'$ partition $E$, so 
\begin{equation*}
\mu(\psi(E)) = \sum_{n=0}^{\infty} \mu(\gamma_n(E_n')) = \int_{E} \exp(\phi(x, \psi(x))) \, d\mu(x)
\end{equation*}
Thus $\frac{d(\mu \circ \psi)}{d\mu}(x) = \exp( \phi(x, \psi(x))$ for $\mu$-a.e. $x \in A$, as required. 
\end{proof}

We will use Lemma \ref{ets_group_conf} in concert with the following lemma, which reduces the question of $(\phi, \fT_X)$-conformality to that of conformality with respect to the finite-order subrelations.

\begin{lemma}\label{fin_imply_full}
Let $\mu$ be a measure on $X$ and $\phi$ a cocycle on $\fT_X$. Suppose that $\mu$ is $(\phi, \mathfrak{T}_{X,N})$-conformal for each $N \geq 0$. Then, $\mu$ is $(\phi, \mathfrak{T}_X)$-conformal.
\end{lemma}

\begin{proof} By Lemma \ref{ets_group_conf}, it is enough to consider only global holonomies. Let $\psi: X \to X$ be a global holonomy of the Gibbs relation $\mathfrak{T}_X$ and let $A \subseteq X$ be a Borel set. We begin by writing $A$ as the increasing union $A = \cup_{N = 0}^{\infty} A_N$, where $A_N = \{ x \in A: (x, \psi(x)) \in \mathfrak{T}_{X,N} \}$. Since $\psi|_{A_N}$ is a holonomy of $\mathfrak{T}_{X,N}$ and $\mu$ is $(\phi, \mathfrak{T}_{X, N})$-conformal, we have 
\begin{align*}
\mu(\psi(A)) &= \lim_{N \to \infty} \mu(\psi(A_N)) \\
&= \lim_{N \to \infty} \int_{A_N} \exp( \phi(x, \psi(x) ) ) \, d\mu(x) \\
&= \int_{A} \exp( \phi(x, \psi(x) ) ) \, d\mu(x),
\end{align*}
by monotone convergence. Thus, $\mu$ is indeed $(\phi, \fT_X)$-conformal by Lemma \ref{ets_group_conf}.
\end{proof}

To echo the comment above on hyperfiniteness, we remark here that both of these results apply, with the same proofs, to any hyperfinite Borel equivalence relation on any Polish space. The following lemma, by contrast, seems to rely more specifically on the structure of $X$ as a subshift.

\begin{lemma}\label{dlr_conf_subrel_abs}
Let $X \subseteq \mathcal{A}^{G}$ be a subshift, let $\phi$ be a cocycle on $X$, and let $\mu$ be a DLR measure on $X$ with respect to $\phi$. Let $N \geq 1$. Then $\mu$ is $(\phi, \fT_{X,N})$-conformal.
\end{lemma}

\begin{proof}
It is enough to show that $\mu(\psi([\omega])) = \int_{[\omega]} \exp( \phi(x, \psi(x))) \, d\mu(x)$ for any cylinder $[\omega]$ and (by Lemma \ref{ets_group_conf}) any global holonomy $\psi$ of $\fT_{X,N}$. Fix a holonomy $\psi: X \to X$ of $\fT_{X, N}$. Since the equivalence classes of $\fT_{X, N}$ are finite, and in fact have bounded cardinality, there exists some $r \geq 0$ such that $\psi^r(x) = x$, for all $x \in X$. Let $m \geq N$ and fix $\omega \in \cA^{\Lambda_m}$. We now partition $X$ according to the orbits of points under $\psi$, in such a way that $[\omega]$ is partitioned into sets that are easy to control. 
Specifically, for each $\overline{\eta} = (\eta_0, \dots, \eta_{r-1}) \in ( \mathcal{A}^{\Lambda_m})^r$, let
\begin{equation*}
T_{\overline{\eta}} = \{ x \in X: \psi^j(x)_{\Lambda_m} = \eta_j, 0 \leq j \leq r-1 \}
\end{equation*}
Note that $T_{\overline{\eta}}$ can be empty. We have $[\omega] = \sqcup_{\overline{\eta} : \eta_0 = \omega } T_{\overline{\eta}}$, and $\psi(T_{\overline{\eta}}) = T_{\overline{\sigma \eta}}$, where $\overline{\sigma \eta} = (\eta_1, \dots, \eta_{r-1}, \eta_0)$ is a cyclic permutation of $\overline{\eta}$. It is enough to show that, for all $\overline{\eta} \in (\mathcal{A}^{\Lambda_m})^r$, we have
\begin{equation*} \mu(\psi(T_{\overline{\eta}})) = \int_{T_{ \overline{\eta}}} \exp \left( {\phi(x, \psi(x))} \right) d\mu(x).
\end{equation*}

By the equality $\psi(T_{\overline{\eta}}) = T_{\overline{\sigma \eta}}$, we have
\begin{equation*}
    \mu(\psi(T_{\overline{\eta}})) = \int_X \mu( T_{\overline{\sigma \eta}} \, | \, \cF_{\Lambda_m^c} ) \, d\mu(x)
\end{equation*}
For any $x \in X$, we know that
\begin{equation*}
    \mathbf{1}_{T_{\overline{\sigma\eta}}}(\eta_1 x_{\Lambda_m^c} )
    = \mathbf{1}_{T_{\overline{\eta}}}(\eta_0 x_{\Lambda_m^c} ) 
\end{equation*}
By this identity, as well as the DLR hypothesis and the defining property of a cocycle, we have the following manipulations:
\begin{align*}
    \mu( T_{\overline{\sigma \eta}} \, | \, \cF_{\Lambda_m^c} )(x) &= \left[ \sum_{\zeta \in \cA^{\Lambda_m}} \exp( \phi(\eta_1 x_{\Lambda_m^c}, \zeta x_{\Lambda_m^c})) \mathbf{1}_X(\zeta x_{\Lambda_m^c})   \right]^{-1} \mathbf{1}_{T_{\overline{\sigma \eta}}}(\eta_1 x_{\Lambda_m^c}) \\
    &= \left[ \sum_{\zeta \in \cA^{\Lambda_m}} \exp( \phi(\eta_0 x_{\Lambda_m^c}, \zeta x_{\Lambda_m^c})) \mathbf{1}_X(\zeta x_{\Lambda_m^c}) \right]^{-1} \\
    & \qquad \qquad \times \mathbf{1}_{T_{\overline{\eta}}}(\eta_0 x_{\Lambda_m^c}) \, \exp( \phi(\eta_0 x_{\Lambda_m^c}, \eta_1 x_{\Lambda_m^c})) \\
    &= \mu( T_{\overline{\eta}} \, | \, \cF_{\Lambda_m^c} )(x) \, \exp( \phi(\eta_0 x_{\Lambda_m^c}, \eta_1 x_{\Lambda_m^c}))
\end{align*}
Integrating this equation yields the result.
\end{proof}

We have therefore done all the work required to prove the following:

\begin{thm}\label{main-thm-abs}
Let $X \subseteq \cA^G$ be a subshift, $\phi$ a cocycle on $X$, and $\mu$ a DLR measure on $X$ with respect to $\phi$. Then $\mu$ is $(\phi, \fT_X)$-conformal.
\end{thm}

\begin{proof}
By Lemma \ref{dlr_conf_subrel_abs}, $\mu$ is $(\phi, \fT_{X,N})$-conformal for each $N$. The result is then immediate from Lemma \ref{fin_imply_full}.
\end{proof}

Theorem \ref{main-thm-abs} was proven by Kimura (\cite{kimura-2015-thesis}, Theorem 5.30) in the special case that $G=\Z^d$, $X$ is a shift of finite type, and the cocycle $\phi$ is induced by a potential, in the manner that we discuss in Proposition \ref{cocyclepot} below. Furthermore, Kimura proved the following converse (\cite{kimura-2015-thesis}, Corollary 5.33), again in the case of $G=\Z^d$ and $\phi$ induced by a potential, but with no finite type assumption on $X$.

\begin{thm}\label{fwd-cocycle}
Let $X \subseteq \cA^G$ be a subshift, $\phi$ a cocycle on $X$, and $\mu$ a $(\phi, \fT_X)$-conformal measure on $X$. Then $\mu$ is DLR with respect to $\phi$.
\end{thm}

The proof of Theorem \ref{fwd-cocycle} is a straightforward adaptation of the methods that Kimura used for the case that he treated. The rough idea is to show that two cylinder sets have conditional measures with the appropriate ratio by considering the holonomy that exchanges them, as in the proof of Lemma \ref{group} above, then applying the conformal hypothesis. The main difference required to adapt the proof is that the version stated here concerns the DLR equations for an arbitrary measurable cocycle, not necessarily one induced by a potential.


\section{Interactions}\label{intrxn_sec}

In this section, we show that, when a cocycle is induced by an interaction, the DLR equations for the cocycle reduce to those for the interaction.

\begin{defn}[interaction]\label{intrxn-defn} An interaction is a family $\Phi = (\Phi_{\Lambda})_{\Lambda \Subset G}$ of functions $\Phi_{\Lambda}: X \to \R$ such that for each $\Lambda \Subset G$, $\Phi_{\Lambda}$ is $\mathcal{F}_{\Lambda}$-measurable, and for all $\Lambda \Subset G$, $x \in X$, the \textit{Hamiltonian series}
    \begin{equation*}
        H_{\Lambda}^{\Phi}(x) = \sum_{\substack{\Delta \Subset G \\ \Delta \cap \Lambda \neq \emptyset}} \Phi_{\Delta}(x)
    \end{equation*}
    converges in the sense that there exists a real number $H_{\Lambda}^{\Phi}(x)$ and, for every $\varepsilon > 0$, there exists some $F \Subset G$ such that, for all $F' \supseteq F$, 
   	\begin{equation*}
   	\left|  H_{\Lambda}^{\Phi}(x) -  \sum_{\substack{\Delta \subseteq F' \\ \Delta \cap \Lambda \neq \emptyset}} \Phi_{\Delta}(x)     \right| < \varepsilon
   	\end{equation*}

\end{defn}

\begin{prop}
Let $\Phi$ be an interaction. For each $(x,y) \in \fT_X$, the series 
\begin{equation*}
\sum_{\Lambda \Subset G} [\Phi_{\Lambda}(x) - \Phi_{\Lambda}(y)]
\end{equation*}
converges in the same sense as the Hamiltonian series. Moreover, the function $\phi_{\Phi}: \fT_X \to \R$ defined by
\begin{equation*}
\phi_{\Phi}(x,y) = \sum_{\Lambda \Subset G} [\Phi_{\Lambda}(x) - \Phi_{\Lambda}(y)]
\end{equation*}
is a cocycle on $\fT_X$.
\end{prop}

\begin{proof}
Let $(x,y) \in \fT_X$ be such that $x_{\Delta^c} = y_{\Delta^c}$. We claim that
\begin{equation*}
\sum_{\Lambda \Subset G} [\Phi_{\Lambda}(x) - \Phi_{\Lambda}(y)] = H^{\Phi}_{\Delta}(x) - H^{\Phi}_{\Delta}(y)
\end{equation*}
with the equality understood in the sense of convergence discussed in the statement of the proposition. Indeed, choose $\varepsilon > 0$. By the definition of an interaction, there exists some $F \Subset G$ sufficiently large that whenever $F \subseteq F' \Subset G$, we have (noting that $\Phi_{E}(x) - \Phi_E(y) = 0$ when $E \cap \Delta = \emptyset$),
\begin{align*}
& \left|  [H_{\Delta}^{\Phi}(x) - H_{\Delta}^{\Phi}(y)] -  \sum_{E \subseteq F'} [\Phi_{E}(x) - \Phi_E(y) ]    \right| \\
\leq & \left|  H_{\Delta}^{\Phi}(x) - \sum_{\substack{E \subseteq F' \\ E \cap \Delta \neq \emptyset}} \Phi_{E}(x) \right| + \left|  H_{\Delta}^{\Phi}(y) -  \sum_{\substack{E \subseteq F' \\ E \cap \Delta \neq \emptyset}}  \Phi_E(y)   \right| \\
< & \, \varepsilon
\end{align*}
This establishes that the series converges, in the sense claimed, to a real number $\phi_{\Phi}(x,y) = H_{\Delta}^{\Phi}(x) - H_{\Delta}^{\Phi}(y)$. Moreover, this energy difference expression makes it obvious that $\phi_{\Phi}$ is a cocycle, concluding the proof.
\end{proof}

We now observe that the DLR equations for the cocycle $\phi_{\Phi}$, in the sense of Definition \ref{dlr-eq-cocycle}, are equivalent to the classical DLR equations for the interaction $\Phi$. Indeed, if $\mu$ is a DLR measure with respect to $\phi_{\Phi}$, then for any $\Lambda \Subset G$, any Borel $A \subseteq X$, and $\mu$-a.e. $x \in X$, we have
\begin{align*}
    \mu(A \,|\, \mathcal{F}_{\Lambda^c})(x) &= \sum_{\zeta \in \cA^{\Lambda}} \left[ \sum_{\eta \in \cA^{\Lambda}} \exp( \phi_{\Phi}(\zeta x_{\Lambda^c}, \eta x_{\Lambda^c})) \mathbf{1}_X(\zeta x_{\Lambda^c})   \right]^{-1} \mathbf{1}_A(\zeta x_{\Lambda^c}) \\
    &= \sum_{\zeta \in \cA^{\Lambda}} \left[ \sum_{\eta \in \cA^{\Lambda}} \exp \left( H_{\Lambda}^{\Phi}(\zeta x_{\Lambda^c}) - H_{\Lambda}^{\Phi}(\eta x_{\Lambda^c})  \right) \mathbf{1}_X(\eta x_{\Lambda^c})   \right]^{-1} \mathbf{1}_A(\zeta x_{\Lambda^c}) \\
    &= \frac{1}{Z_{\Lambda}^{\Phi}(x)} \sum_{\zeta \in \cA^{\Lambda}}  \exp \left( - H_{\Lambda}^{\Phi}(\zeta x_{\Lambda^c})  \right) \mathbf{1}_A(\zeta x_{\Lambda^c})
\end{align*}
where
\begin{equation*}
    Z_{\Lambda}^{\Phi}(x) = \sum_{\eta \in \cA^{\Lambda}} \exp \left(  - H_{\Lambda}^{\Phi}(\eta x_{\Lambda^c})  \right) \mathbf{1}_X(\eta x_{\Lambda^c})
\end{equation*}
By Theorem \ref{main-thm-abs}, if $\mu$ satisfies these (classical) DLR equations for $\Phi$, then $\mu$ is $(\phi_{\Phi}, \fT_X)$-conformal.


\section{Potentials}\label{potl_sec}

In this section and the next, we restrict to finitely generated groups $G$ satisfying a certain growth condition. We need this condition in order to construct a cocycle from a potential in a way that is compatible with interactions, in a sense to be made precise in \S\ref{potl_from_intrxn_sec}. The condition is as follows. It concerns the \textit{spherical growth function} $|B_k \setminus B_{k-1}|$, which is a basic quantity studied in geometric group theory, discussed for instance in (\cite{delaharpe2000groups}, \S VI.A).

\begin{defn}[bounded sphere ratios]
Let $G$ be a finitely generated group. With respect to a finite generating set $S \Subset G$, we can consider the open balls $B_k = \{ g \in G \, : \, d(g,e) < n   \}$ of radius $k$ centered at the identity in the Cayley graph of $G$ with respect to $S$. We say that a group $G$ has \textit{bounded sphere ratios} if there exists a finite generating set $S$ such that 
\[
\sup_{m \geq 1} \frac{|B_{m+1} \setminus B_{m}|}{|B_m \setminus B_{m-1}|} < +\infty.
\]
\end{defn}

In this section and the next, when we refer to balls in a group $G$ with bounded sphere ratios, we always mean balls with respect to a generating set that witnesses the bounded sphere ratios. Note also that if $G$ has bounded sphere ratios, then (for some generating set $S$) we have
\begin{equation*}
   \sup_{m \geq 1} \frac{|B_{m+n} \setminus B_{m+n-1}|}{|B_m \setminus B_{m-1}|} < +\infty.
\end{equation*}
for any $n$.

\begin{rmk}
A finitely generated group $G$ has \textit{polynomial growth} if $|B_n| \leq cn^d$ for some $c > 0, d \in \N$ and all $n$; \textit{exponential growth} if $|B_n| \geq  c\alpha^n$ for some $\alpha > 1, c > 0$ and all $n$; and \textit{intermediate growth} otherwise. Here we outline certain types of polynomial and exponential growth known to imply bounded sphere ratios.

In the polynomial case, recall that a group has polynomial growth  if and only if it is virtually nilpotent, i.e. has a finite-index nilpotent group \cite{gromov1981groups}. It is conjectured (\cite{breuillard-le-donne-2013-rate}, Conjecture 10) that for any nilpotent group, we have $|B_n| = cn^d + O(n^{d-1})$, where $c>0$ is a constant depending only on the group, with the coefficients of the lower-order terms depending on the generating set. This would imply (\cite{breuillard-le-donne-2013-rate}, Corollary 11) positive constant upper and lower bounds on the ratio $|B_k \setminus B_{k-1}|/k^{d-1}$, and thus that the group has bounded sphere ratios. What is known is more restricted. Associated to any nilpotent group is its \textit{nilpotency class}, a number measuring how far the group is from being abelian (abelian groups, like $\Z^d$, have class $1$). By a result of Stoll \cite{stoll-1998-two-step}, the conjectured asymptotics for $|B_n|$ hold at least for groups of nilpotency class at most $2$.

In the exponential case, we say that a group (with a given generating set) has \textit{exact exponential growth} if there exist $\alpha > 1$ and $0 < c < C < c\alpha$ with $c \leq |B_n|/\alpha^n \leq C$ for all $n \geq 1$. (This is not a standard definition.) This condition is satisfied, for example, by a free group with the usual generating set. To see that exact exponential growth implies bounded sphere ratios, note that
\[
\frac{|B_{k+1} - B_k|}{|B_k - B_{k-1}|} \leq \frac{C\alpha^{k+1} - c \alpha^k}{c \alpha^k - C \alpha^{k-1}} 
= \left( \frac{C\alpha - c}{c \alpha - C} \right) \alpha < \infty
\]
\vspace{1em}
\end{rmk}

We now turn our attention to potentials. For a function $f: X \to \R$ and $k \geq 1$, define the $k$th variation of $f$ as
\begin{equation*}
    v_{k} (f) := \sup \left\{ | f(y) - f(x) | \, \Big| \,  x, y \in X, \, x_{B_{k}} = y_{B_{k}}  \right\}.
\end{equation*}
We separately define $v_{0}(f) = \| f \|_{\infty}$. It is also convenient to define $B_{0} = \emptyset$. We define the \textit{shell norm} $\ShVar{\cdot}$ by
\begin{equation*}
    \ShVar{f} := \sum_{k = 0}^{\infty} | B_{k+1} \setminus B_k | v_{k} (f).
\end{equation*}
We define the space $\ShReg(X)$ as the space of \textit{shell-regular potentials}, i.e., functions $f:X \to \R$ with $\ShVar{f} < \infty$. It is elementary to show that shell-regularity implies continuity, and that $\ShReg(X)$, with the shell norm, is a Banach space. Note that this space depends, in general, on the generating set chosen.

\begin{rmk} In earlier work on subshifts over $\Z^d$ \cite{meyerovitch-2013-gibbs-eqm}, the relevant space of potentials is known as $\mathrm{SV}_d(X)$, the space of potentials with $d$-summable variation, defined by the norm $\| f \|_{SV_d} = \sum_{k=1}^{\infty} k^{d-1} v_{k-1}(f)$. This space is also known as $\mathrm{Reg}_{d-1}(X)$ \cite{muir2011gibbs}. With $B_n = \Z^d \cap [-n,n]^d$, we have $|B_{k+1} \setminus B_{k}| = 2^d d (1 + o(1)) k^{d-1}$. Thus, on $\Z^d$, we have $\ShReg(X) = \mathrm{SV}_d(X)$, with the identity a continuous linear map.
\end{rmk}

\begin{prop}\label{cocyclepot}
Let $G$ be a group with bounded sphere ratios and $X$ a subshift over $G$. For any $f \in \ShReg(X)$ and any $(x,y) \in \fT_X$, the series
\begin{equation*}
\sum_{g \in G} [f(g \cdot y) - f(g \cdot x)]
\end{equation*}
converges absolutely and defines a cocycle $\phi_f$ on $\fT_X$.
\end{prop}

\begin{proof}
Fix $(x,y) \in \fT_X$, and let $n \geq 1$ be such that $x_{B_n^c} = y_{B_n^c}$. If $g \in G$ and $m \geq 1$ are such that $B_{m-1} \subseteq g^{-1} B_n^c$, then $(g \cdot x)|_{B_{m-1}} = (g \cdot y)|_{B_{m-1}}$ so $|f(g \cdot y) - f(g \cdot x)| \leq v_{m-1}(f)$. For $m \geq 1$ and $g \in B_k \setminus B_{k-1}$, the triangle inequality guarantees that $g B_{m-1} \subseteq B_n^c$ if $k-n \geq m$. Since the shells $B_{k+1} \setminus B_k$ partition $G$, we have
\begin{align*}
\sum_{g \in G} | f(g \cdot y) - f(g \cdot x)| &\leq 2 | B_n | \| f \|_{\infty} +  \sum_{k=n+1}^{\infty} | B_k \setminus B_{k-1} | v_{k-n - 1}(f) \\
&\leq 2 | B_n | \| f \|_{\infty} + \left( \sup_{k \geq 1} \frac{|B_{k+n} \setminus B_{k+n-1}|}{|B_{k} \setminus B_{k-1}|} \right) \ShVar{f}
\end{align*}
so indeed the cocycle is well-defined by an absolutely convergent series.
\end{proof} 


Just as in the case of an interaction, this expression for the cocycle $\phi_f$ allows us to rewrite the DLR equations in a more classical form. Let $f \in \ShReg(X)$. It follows from a simple manipulation that for any $(x, y) \in \fT_X$, we have
\begin{equation*}
    \exp(\phi_f(x,y)) = \lim_{m \to +\infty} \exp \left(\sum_{g \in B_m} [f(g \cdot y) - f(g \cdot x)] \right) = \lim_{m\to +\infty} \frac{\exp f_m(y)}{\exp f_m(x)}.
\end{equation*}
where $f_m(z) = \sum_{g \in B_m} f(g \cdot z)$. Now, let $A \subseteq X$ be a Borel set. If $\mu$ is a DLR measure with respect to $\phi_f$, then for $\mu$-a.e. $x \in X$, we have
\begin{align*}
    \mu(A \, | \, \cF_{\Lambda^c}) &= \sum_{\eta \in \cA^{\Lambda}} \left[ \sum_{\zeta \in \cA^{\Lambda}} \exp( \phi_f(\eta x_{\Lambda^c}, \zeta x_{\Lambda^c})) \mathbf{1}_X(\zeta x_{\Lambda^c})   \right]^{-1} \mathbf{1}_A(\eta x_{\Lambda^c}) \\
    &= \sum_{\eta \in \cA^{\Lambda}} \left[ \sum_{\zeta \in \cA^{\Lambda}} \lim_{m\to +\infty} \frac{\exp f_m(\zeta x_{\Lambda^c})}{\exp f_m(\eta x_{\Lambda^c})}\mathbf{1}_X(\zeta x_{\Lambda^c})   \right]^{-1} \mathbf{1}_A(\eta x_{\Lambda^c})\\
    &= \lim_{m \to \infty} \frac{ \sum_{\eta \in \mathcal{A}^{\Lambda}} \exp \left( {f_m(\eta x_{\Lambda^c})} \right)
\mathbf{1}_A(\eta x_{\Lambda^c} )}
{\sum_{\zeta \in \mathcal{A}^{\Lambda} } \exp \left( {f_m(\zeta x_{\Lambda^c})} \right) \mathbf{1}_X(\zeta x_{\Lambda^c}) }
\end{align*}
These are the DLR equations as found in Kimura \cite{kimura-2015-thesis}. Applying Theorem \ref{main-thm-abs} therefore shows that any DLR measure with respect to a potential $f \in \ShReg(X)$ is necessarily $(\phi_f, \fT_X)$-conformal, providing the full converse for Kimura's result described in the introduction.


\section{Potentials induced by interactions, and vice versa}\label{potl_from_intrxn_sec}

We have seen that the DLR property implies the conformal property for an arbitrary cocycle on the Gibbs relation, with Gibbs measures for interactions and for potentials as two special cases. These cases are not independent. In this section, we adapt the methods and results of Muir \cite{muir2011gibbs} and Ruelle \cite{ruelle-2004-thermo} to construct potentials from interactions and vice versa. In this section, all interactions are translation-invariant, i.e., for any $\Lambda \Subset G$ and any $x \in X$, we require that $\Phi_{g \Lambda} (g \cdot x) = \Phi_{\Lambda}(x)$. We recall a classical space of particularly well-behaved interactions:

\begin{defn}
For an interaction $\Phi$, let
\begin{equation*}
\| \Phi \|_B = \sum_{\substack{\Lambda \Subset G \\ e \in \Lambda}} \| \Phi_{\Lambda} \|_{\infty}
\end{equation*}
We define $\cB$ as the space of \textit{absolutely summable} $\Phi$, i.e., those for which $\| \Phi \|_B < \infty$.
\end{defn}

It is routine to check that $(\cB, \| \cdot \|_B)$ is a Banach space. Moreover, for$\Phi \in \cB$, we in fact have absolute convergence of the series defining the cocycle $\phi_{\Phi}$, since for any $(x,y) \in \fT_X$ with $x_{\Delta^c} = y_{\Delta^c}$ for some $\Delta \Subset G$, we have
\begin{align*}
\sum_{\Lambda \Subset G} | \Phi_{\Lambda}(x) - \Phi_{\Lambda}(y)| &\leq 2 \sum_{\substack{\Lambda \Subset G \\ \Lambda \cap \Delta \neq \emptyset}} \| \Phi_{\Lambda} \|_{\infty} \\
&\leq 2 | \Delta | \sum_{\substack{\Lambda \Subset G \\ e \in \Lambda}} \| \Phi_{\Lambda} \|_{\infty} \\
&= 2 | \Delta | \| \Phi \|_B < \infty
\end{align*}

We introduce a family of linear maps that convert interactions into potentials.

\begin{defn}[translate-weighting maps]
Let $(a_{\Lambda})_{\Lambda \Subset G, \, e \in \Lambda}$ be a collection of nonnegative real coefficients such that, for each $\Lambda \Subset G$ with $e \in \Lambda$, we have $\sum_{g \in \Lambda} a_{g^{-1} \Lambda} = 1$. Then, for an interaction $\Phi$, define the potential $A_{\Phi}$ via 
\begin{equation*}
A_{\Phi}(x) = - \sum_{\substack{\Lambda \Subset G \\ e \in \Lambda}} a_{\Lambda} \Phi_{\Lambda}(x)
\end{equation*}

The map $\Phi \mapsto A_{\Phi}$ is clearly linear. We refer to this map as the translate-weighting map determined by the weights $(a_{\Lambda})_{\Lambda \Subset G, e \in \Lambda}$.
\end{defn}

\begin{rmk}
Two important examples are the following.
\begin{itemize}
\item The \textit{uniform map}, where $a_{\Lambda} \in 
\{ 0, \frac{1}{|\Lambda|} \}$ for every nonempty $\Lambda \Subset G$. Muir uses the letter $A$ to denote this specific operator, i.e., $A(\Phi) = A_{\Phi}$.

\item The class of \textit{dictator maps}, where $a_{\Lambda} \in \{ 0, 1 \}$ for every $\Lambda \Subset G$. For instance, on $\Z^d$, Ruelle studies the operator for which $a_{\Lambda} = 1$ if and only if $0$ is the middle element, or more precisely the $\lfloor (|\Lambda| + 1)/2 \rfloor$-th element, of $\Lambda$ in lexicographic order. In \cite{muir2011gibbs}, Muir refers to this operator as $\hat{A}$.
\end{itemize}

\end{rmk}

In Fact 7.8 in \cite{muir2011gibbs}, it is claimed that $A_{\Phi} \in \ShReg(X)$ for every translate-weighting map and every $\Phi \in \cB$. This claim is incorrect, as we demonstrate with an example below. However, the argument presented for this claim is correct in the case of what Muir calls ``cubic-type'' interactions. Here we reproduce a version of this proof for a larger class of interactions.

\begin{defn}
An interaction $\Phi$ is \textit{full-dimensional} if there exists some $C > 0$ such that, for all $\Lambda \Subset G$ with $e \in \Lambda$ and $\Phi_{\Lambda} \not\equiv 0$, we have the bound
\begin{equation*}
\sup \{  |B_n| : \, n \in \N, \, \Lambda \cap B_{n-1}^c \neq \emptyset   \} \leq C |\Lambda|
\end{equation*}
\end{defn}

\begin{prop}\label{full-dim-shreg} Let $G$ be a group with bounded sphere ratios and let $X$ be a subshift over $G$.
If $\Phi \in \cB$ is full-dimensional, then $A_{\Phi} \in \ShReg(X)$, where $A_{\Phi}$ is the image of $\Phi$ under an arbitrary translate-weighting map.
\end{prop}

\begin{proof}
We first estimate $v_{k-1}(A_{\Phi})$:
\begin{align*}
v_{k-1}(A_{\Phi}) &= \sup \left\{ \left| \sum_{\substack{\Lambda \Subset G \\ e \in \Lambda} } a_{\Lambda} [ \Phi_{\Lambda}(x) - \Phi_{\Lambda}(y) ] \right| \, : \, x, y \in X, \, x_{B_{k-1}} = y_{B_{k-1}}  \right\} \\
%
%
&\leq 2 \sum_{\substack{\Lambda \Subset G \\ e \in \Lambda \\ \Lambda \cap B_{k-1}^c \neq \emptyset} } a_{\Lambda} \| \Phi_{\Lambda} \|_{\infty}
\end{align*}

We can now estimate the shell norm by an exchange of summations:
\begin{align*}
\ShVar{A_{\Phi}} &\leq 2 \sum_{k=0}^{\infty} |B_{k+1} \setminus B_{k} | \sum_{\substack{\Lambda \Subset G \\ e \in \Lambda \\ \Lambda \cap B_{k}^c \neq \emptyset} } a_{\Lambda} \| \Phi_{\Lambda} \|_{\infty} \\
&= 2 \sum_{\substack{ \Lambda \Subset G \\ e \in \Lambda}} a_{\Lambda} \| \Phi_{\Lambda} \|_{\infty}  \sum_{\substack{k \geq 0 \\ \Lambda \cap B_{k}^c \neq \emptyset}} |B_{k+1} \setminus B_{k}|  
\end{align*}

Observe that 
\begin{equation*}
\sum_{\substack{k \geq 0 \\ \Lambda \cap B_{k-1}^c \neq \emptyset}} |B_{k+1} \setminus B_{k}|  = \sup \{  |B_n| : \, n \in \N, \, \Lambda  \cap B_{n}^c \neq \emptyset   \} \leq C | \Lambda|
\end{equation*}
so in fact
\begin{equation*}
\ShVar{A_{\Phi}} \leq 2C \sum_{\substack{ \Lambda \Subset G \\ e \in \Lambda}}  a_{\Lambda} |\Lambda| \| \Phi_{\Lambda} \|_{\infty}
\end{equation*}

We need to rearrange this sum. For a given $\Lambda \Subset G$, consider the set of translates of $\Lambda$ containing the identity, denoted $T(\Lambda) = \{ g^{-1} \Lambda, \, g \in \Lambda \}$. For instance, in $\Z$, if $\Lambda=\{0,1\}$, then $T(\Lambda)= \{ \{ -1,0 \}, \{ 0,1 \}  \}$. Let $\cT$ denote the set of such sets of translates, i.e., $\cT = \{ T(\Lambda) \, : \, \Lambda \Subset G, \, e \in \Lambda \}$. Note that $\cT$ is a partition of the set $\{ \Lambda \Subset G, \, e \in \Lambda  \}$. Observe furthermore that $|T| = |\Lambda|$ for any $\Lambda \in T$.

For any given $T \in \cT$, the value $|\Lambda| \| \Phi_{\Lambda} \|_{\infty}$ is the same for any $\Lambda \in T$, i.e., any $\Lambda$ such that $T = T(\Lambda)$. so we denote it by $c_T$. We can then express the bound on $\ShVar{A_{\Phi}}$ by summing over $T \in \cT$, as follows:
\begin{align*}
\sum_{\substack{ \Lambda \Subset G \\ e \in \Lambda}} a_{\Lambda} |\Lambda|  \| \Phi_{\Lambda} \|_{\infty} &= \sum_{T \in \cT} \sum_{\Lambda \in T} a_{\Lambda}  c_T \\
&= \sum_{T \in \cT} c_T \sum_{\Lambda \in T} a_{\Lambda}  \\
&= \sum_{T \in \cT} c_T \\
&= \sum_{T \in \cT} |\Lambda| \| \Phi_{\Lambda} \|_{\infty} \\
&= \sum_{T}  \sum_{\Lambda \in T} \| \Phi_{\Lambda} \|_{\infty} \\
&= \| \Phi \|_B
\end{align*}

Thus $\ShVar{A_{\Phi}} \leq 2C \| \Phi \|_B < \infty$. 
\end{proof}

The following example shows that if $\Phi \in \cB$ is not full-dimensional, then $A_{\Phi}$ can fail to be shell-regular.
\begin{example}
Let $X = \{0,1\}^{\mathbb{Z}}$, with $B_k=(-k,k) \cap \Z$. Define $\Phi = (\Phi_{\Lambda})_{\Lambda \Subset \Z}$ as follows: for any $i, j \in \mathbb{Z}$, $\Phi_{\{ i,j \} }(x) = \frac{1}{(j-i)^2}$ if $x_i = x_j = 1$ and $0$ otherwise; and $\Phi_{\Lambda} \equiv 0$ for all other $\Lambda \Subset G$. Clearly $\Phi$ is translation-invariant. We claim that $\Phi \in \cB$ but $A_{\Phi} \notin \ShReg(X)$, where $A_{\Phi}$ is the image of $\Phi$ under the dictator map that ignores $\Lambda \Subset \Z$ unless $0 = \inf \Lambda$. Indeed, $\| \Phi \|_{\mathcal{B}} = 2 \sum_{j=1}^{\infty} \frac{1}{j^2} < \infty$, but
\begin{equation*}
    v_k(A_{\Phi}) = \sum_{l = k}^{\infty}\frac{1}{l^2}\geq \frac{1}{k}
\end{equation*}
which implies that
\begin{equation*}
\ShVar{A_{\Phi}} \geq 2 \sum_{k=1}^{+\infty} \frac{1}{k} = +\infty
\end{equation*}

\end{example}


The next two propositions establish that for any full-dimensional interaction $\Phi \in \cB$, the images $A_{\Phi}$ and $A'_{\Phi}$ of $\Phi$ under any two translate-weighting maps have all the same Gibbs and equilibrium measures.

\begin{prop}\label{same_cocycle} Let $G$ be a group with bounded sphere ratios, let $X$ be a subshift on $G$, and let $\Phi$ be an absolutely summable, full-dimensional interaction on $X$. Then $\Phi$ and $A_{\Phi}$ induce the same cocycle, i.e., $\phi_{A_{\Phi}} = \phi_{\Phi}$, where $A_{\Phi}$ is the image of $\Phi$ under an arbitrary translate-weighting map.
\end{prop}


\begin{proof}
Suppose that $(x,y) \in \fT_X$ with $x_{\Delta^c} = y_{\Delta^c}$. Observe that
\[
\phi_{\Phi}(x,y) = \sum_{\substack{\Lambda \Subset G \\ \Lambda \cap \Delta \neq \emptyset}} \left[ \Phi_{\Lambda}(x) - \Phi_{\Lambda}(y) \right]
\]
Consider a translate-weighting map with weights $a_{\Lambda}$. To compute $\phi_{A_{\Phi}}$, we first obtain a convenient expression for $A_{\Phi}(g \cdot x) - A_{\Phi}(g \cdot y)$:
\begin{align*}
    A_{\Phi}(g \cdot x) - A_{\Phi}(g \cdot y) &= - \sum_{\substack{\Lambda \Subset G \\ e \in \Lambda \\ \Lambda \cap g \Delta \neq \emptyset}} a_{\Lambda} \Phi_{g^{-1} \Lambda} (x) +  \sum_{\substack{\Lambda \Subset G \\ e \in \Lambda \\ \Lambda \cap g \Delta \neq \emptyset}} a_{\Lambda} \Phi_{g^{-1} \Lambda} (y)\\
    &= - \sum_{\substack{\Lambda' \Subset G \\ g \in \Lambda' \\ \Lambda' \cap \Delta \neq \emptyset}} a_{g^{-1} \Lambda'} [\Phi_{\Lambda'} (x)  - \Phi_{\Lambda'} (y)]
\end{align*}
We then compute:
\begin{align*}
\phi_{A_{\Phi}}(x,y) &= \sum_{g \in G} [A_{\Phi}(g \cdot x) - A_{\Phi}(g \cdot y)]\\
&= \sum_{g \in G} \sum_{\substack{\Lambda \Subset G \\ g \in \Lambda \\ \Lambda \cap \Delta \neq \emptyset}}  a_{g^{-1} \Lambda} [ \Phi_{\Lambda}(x) - \Phi_{\Lambda} (y) ] \\
&= \sum_{\substack{\Lambda \Subset G\\ \Lambda \cap \Delta \neq \emptyset}} \left( \sum_{g \in \Lambda} a_{g^{-1} \Lambda} \right) [ \Phi_{\Lambda}(x) - \Phi_{\Lambda} (y) ] \\
&= \phi_{\Phi}(x,y)
\end{align*}

The interchange of summations is justified by the absolute convergence of the series defining the cocycles $\phi_{A_{\Phi}}$ and $\phi_{\Phi}$, implied by the regularity of $\Phi$ and $A_{\Phi}$.
\end{proof}

Proposition \ref{same_cocycle} is similar to Theorem 5.42 in \cite{kimura-2015-thesis}, which is stated for Ruelle's operator $A$, using specifications rather than cocycles.

\begin{prop}
Let $G$ be a group with bounded sphere ratios and let $X$ be a subshift on $G$. Let $\mu$ be a $G$-invariant measure on $X$, let $\Phi \in \cB$ be full-dimensional. Let $A_{\Phi}$ be the image of $\Phi$ under a translate-weighting map with weights $(a_{\Lambda})_{\Lambda \Subset G, \, e \in \Lambda}$. Then the integral $\int_X A_{\Phi} \, d\mu$ depends only on $\Phi$ and $\mu$, and not on the weights $a_{\Lambda}$.
\end{prop}

\begin{proof}
As in the proof of Proposition \ref{full-dim-shreg}, for each finite $\Lambda \Subset G$ with $e \in \Lambda$, let $T(\Lambda) = \{ g^{-1} \Lambda \, | \, g \in \Lambda \}$. For any given $T$, the quantity $\int_X \Phi_{\Lambda} \, d\mu$ is constant as $\Lambda$ ranges over $T$, so we denote it by $b_T$. We now compute:
\begin{align*}
\int_X A_{\Phi} \, d\mu &= -\int_X \sum_{T \in \cT} \sum_{\Lambda \in T} a_{\Lambda} \Phi_{\Lambda} \, d\mu \\
&= -\sum_{T \in \cT} b_T \sum_{\Lambda \in T} a_{\Lambda}   \\
&= -\sum_{\substack{\Lambda \Subset G \\ e \in \Lambda}} \frac{1}{|\Lambda|} \int_X \Phi_{\Lambda} \, d\mu
\end{align*}
which does not depend on the weights $a_{\Lambda}$, and in addition clearly expresses the integral $\int_X A_{\Phi} \, d\mu$ as the average energy at the identity due to the interaction $\Phi$.

To justify exchanging the integral and the sum above, let $|\Phi|$ be the interaction given by $|\Phi|_{\Lambda} = |\Phi_{\Lambda}|$. Then $|\Phi|$ is still full-dimensional, with $\| |\Phi| \|_B = \|  \Phi  \|_B$, so 
\begin{equation*}
    -\sum_{T \in \cT} \sum_{\Lambda \in T} a_{\Lambda} |\Phi_{\Lambda}| = A_{|\Phi|} \in \ShReg(X)
\end{equation*}
by Proposition \ref{full-dim-shreg}. Thus the sum converges absolutely to a continuous function.
\end{proof}

Finally, we introduce a smaller Banach space $\VolReg(X)$ of \textit{volume-regular functions}, defined analogously to $\ShReg(X)$ by a volume norm rather than a shell norm. That is, $\VolReg(X) = \{ f: X \to \R \, : \, \VolVar{f} < \infty  \}$ where we define
\begin{equation*}
\VolVar{f} := \sum_{k = 0}^{\infty} | B_k | v_{k} (f)
\end{equation*}
Volume-regularity clearly implies shell-regularity. The following result of Muir (\cite{muir2011gibbs}, proof of Fact 7.6) is stated for $\Z^d$, with the name $\mathrm{Reg}_d(X)$ for $\VolReg(X)$, but is valid, with the same proof, on any finitely generated group.

\begin{thm}\label{preimg}
Let $G$ be a finitely generated group and let $f \in \VolReg(X)$ be a volume-regular potential. Then there exists an absolutely summable $\Phi \in \cB$ with $A_{\Phi} = f$ where $A_{\Phi}$ is the image of $\Phi$ under some dictator map.
\end{thm}

In particular, any Gibbs measure for $f \in \VolReg(X)$ is also a Gibbs measure for any potential $\Phi \in \cB$ with $A_{\Phi} = f$, and vice versa. 


\section*{Acknowledgments}

We thank Brian Marcus and Tom Meyerovitch for many helpful and generous discussions throughout the course of this work. We also thank Rodrigo Bissacot for drawing our attention to the work of Kimura; Bruno Kimura and Sebasti{\'a}n Barbieri for helpful conversations; Nishant Chandgotia for providing the example in \S\ref{potl_from_intrxn_sec}; and an anonymous referee for several helpful suggestions. LB is supported by grants 2018/21067-0 and 2019/08349-9, S\~ao Paulo Research Foundation (FAPESP).

\end{document}